\documentclass[12pt]{iopart}

\usepackage{graphicx}
\usepackage{iopams}
\usepackage{xcolor}

\usepackage{verbatim}
\usepackage{amsthm}

\newcommand{\RR}{\mathcal{R}}
\newcommand{\R}{\mathbb{R}}
\newcommand{\TT}{\mathcal{T}}
\newcommand{\NN}{\mathbb{N}}

\newtheorem{theorem}{Theorem}[section]

\newtheorem{lemma}{Lemma}[section]
\newtheorem{corollary}{Corollary}[section]

\graphicspath{{./Figures/},{./Figures/Illustrations/},{./Data2/},{./Figures_new2/}} 

\begin{document}

\title{On a three-dimensional Compton scattering tomography system with fixed source}

\author{J Cebeiro$^1$, C Tarpau$^{2,3,4}$, M A Morvidone$^1$, D Rubio$^1$ and M K Nguyen$^2$}

\address{$^1$ Centro de Matem\'atica Aplicada, Universidad Nacional de San Mart\'in (UNSAM), Buenos Aires, Argentina.}

\address{$^2$ Equipes Traitement de l'Information et Syst\`emes, CY Cergy Paris University, ENSEA, CNRS UMR 8051, Cergy-Pontoise, France.}

\address{$^3$ Laboratoire de Physique Th\'eorique et Mod\'elisation, CY Cergy Paris University, CNRS UMR 8089, Cergy-Pontoise, France.}

\address{$^4$ Laboratoire de Math\'ematiques de Versailles, Universit\'e de Versailles Saint-Quentin, CNRS UMR 8100, Versailles, France.}

\ead{jcebeiro@unsam.edu.ar}
\vspace{10pt}

\begin{abstract}  

Compton scattering tomography is an emerging scanning technique with attractive applications in several fields such as non-destructive testing and medical imaging. In this paper, we study a modality in three dimensions that employs a fixed source and a single detector moving on a spherical surface. We also study the Radon transform modeling the data that consists of integrals on toric surfaces. Using spherical harmonics we arrive to a generalized Abel's type equation connecting the coefficients of the expansion of the data with those of the function. We show the uniqueness of its solution and so the invertibility of the toric Radon transform. We illustrate this through numerical reconstructions in three dimensions using a regularized approach. A preliminary version of the algorithm for discrete spherical harmonic expansion is available in the code repository \cite{Cebeiro_Algorithm_for_Discrete_2020}.
\end{abstract}

%
%
%
%
%

\section{Introduction}

Compton scattering imaging takes advantage of scattered radiation in order to explore the interior of an object under study. The basis of Compton scattering imaging is Compton effect where a photon of energy $E_0$ interacts with an electron undergoing a change in its direction and losing part of its of energy. The angular deviation $\omega$ experienced by the photon and its remaining energy $E(\omega)$ are related by the well-known Compton formula

\begin{equation*}
\label{eq:compton}
E(\omega)=\frac{E_0}{1+ \frac{E_0}{mc^2}(1-\cos \omega)},
\end{equation*}
\noindent
where $E_0$ is the energy of the incident photon and $mc^2$ is the energy of the electron at rest (511 keV). Transmission Compton scattering tomography consists in illuminating the object under study with a source of radiation and registering the scattered photons in its neighborhood. In this case the function to recover is a map of the electronic density of the object \cite{NS1994}. The interest in Compton scattering imaging is supported by its numerous potential applications namely non-destructive testing \cite{CNMN17}, airport security \cite{Webber_2020} and the study of composite materials, for instance those used in aircraft industry \cite{tarpau2019ndt}. Recent works in medical imaging suggest potential advantages in diagnosis since images may exhibit better contrast in certain scenarios such as lung tumors \cite{jones2018characterization, redler2018compton}. 

The first formulation of Compton tomography based on a Radon type transform has been introduced by Norton in \cite{NS1994}. The mathematical foundations arising from the seminal works of Allan M. Cormack on circular Radon transforms provide a formal framework to study these modalities \cite{Cormack81,Cormack82,Cormack84}. Depending on the setup, forward models are based on integral transforms on different manifolds. In a two dimensional configuration, where collimated sources are used to restrict photons to a given scanning plane, manifolds are circular arcs \cite{NS1994,CNMN17,TN2011,TN2011a,INTECH,webber_2016,Norton2019,TCMN2019}. If uncollimated detectors are used, manifolds consist in toric sections \cite{Webber_2020,webber_quinto_2019microlocal,TCNRM2020}. 

In three-dimensional Compton scattering tomography, where uncollimated source and detectors are employed, photons recorded at a given place and energy carry information about the electronic density all over a toric surface. In these setups scanning is performed in a volume rather than in a plane and the stacking of planar slices after reconstruction is no longer necessary. The problem in three dimensions has recently been addressed in \cite{RH2018} where a comprehensive framework was introduced together with four three-dimensional modalities. This framework takes attenuation into account and provides a back-projection algorithm for contour reconstruction. For some of these configurations, there are also results on injectivity, uniqueness and numerical reconstructions of the electronic density \cite{WL2018,Webber_2020}. 

The design of an operating Compton scattering tomography system raises numerous challenges involving size, detectors, sources, shielding and precise mechanisms to move components, sometimes in a synchronous fashion. In this respect, the suppression of the need of synchronized movements, the ability to scan a large object without surrounding it and the reduction of the number of moving components are advantageous design features. For instance, employing one fixed source turns valuable, particularly when using electrically powered sources that are expected to be more difficult to handle because of its weight and wiring. 

Here, we study a Compton scattering tomography in three dimensions with one fixed central source and unique detector moving on a sphere. In particular, we show the invertibility of the toric transform modeling data acquisition. Instead of forcing the object under study to be strictly inside the detection sphere \cite{RH2018}, we placed it strictly outside. Through this relative positioning of the specimen, the scanning can be carried out using a compact system without the need of enclosing or surrounding the object. This makes the configuration suitable for scanning large bodies or objects attached to major structures. Recently, a two-dimensional reduction of this model has been studied in \cite{TCNRM2020}. This reduction enabled to solve the reconstruction problem through an analytic reconstruction formula. Nevertheless, the setup requires translational relative movement between the object and the source in order to scan planar sections to be stacked in a three-dimensional reconstruction. There are 3D CST modalities for which object reconstruction has been shown feasible by means of the invertibility of its Radon transform. Two configurations, both registering back-scattered photons \cite{WL2018, Webber_2020}, are suitable for scanning an object without surrounding it. While \cite{WL2018} requires the rotation of the pair source-detector, the design in \cite{Webber_2020} requires translational movement of the source that can be alternatively accomplished employing an array of multiple sources. Regarding the performance in terms of photon counts, the efficiency of the configuration studied here may be comparable to these setups since all of them employ the same type of photons (back-scattered). There is also a setup in \cite{WL2018} that employs forward-scattered photons to scan a body inside the sphere enclosed by the source-detector pair. In this case, comparative efficiency depends upon the range of energies considered.


In order to model the forward problem, we formulate a Radon transform on tori and study its properties, particularly those connected to the uniqueness of its solution. The paper is organized as follows. In section \ref{sec:model} we explain briefly the setup, the scanning protocol and the toric Radon transform used to model image formation. We also write the transform as a spherical harmonics expansion and obtain an  explicit Abel's type equation relating the coefficients of the expansion of the data with those of the function that represents the electronic density of the sought function. In section \ref{sec:uniqueness}, we analyze its kernel and show the uniqueness of its solution proving, thus, the invertibility of the toric Radon transform. Numerical reconstructions based on discrete spherical harmonics and Tikhonov regularization are shown in section \ref{sec:NS}. We discuss the results and the specific advantages and limitations of the modality as well as some new perspectives of our work in section \ref{sec:discusion}. Finally, we close the paper with some conclusions. Details of the spherical harmonics expansion and the discrete algorithm we used to implement it are given in \ref{sec:SH_reduction} and \ref{sec:DSHE}.


\section{A setup for three dimensional Compton scattering tomography}

\label{sec:model}

We introduce a setup for Compton scattering tomography based on a central source of radiation $S$ located at the origin of coordinates. A single detector $D$ moves at a constant distance $R$ of the origin describing a sphere, as shown in figure \ref{fig:setup}. The property to recover is the electronic density $f(x,y,z)$ of an object placed strictly outside of this spherical surface. Detector $D$ registers backscattered photons of energy $E(\omega)$ that have interacted with electrons of the object and deviated an angle $\omega$ from its original path. When detector $D$ is at position labeled by angles $(\alpha,\beta)$ and registers a photon of energy $E(\omega)$, the interaction site is located somewhere on the surface of an apple torus through points $S$ and $D$, which is characterized by $\alpha,\,\beta$ and $\omega$. Thus, the flux of photons of energy $E(\omega)$ registered by $D$ at sites $(\alpha,\beta)$ is proportional to the integral of function $f$ on a toric surface $\mathcal{T}^{\omega,\alpha,\beta}$, with $\alpha \in [0,2\pi)$, $\beta \in [0,\pi]$ and $\omega\in(\pi/2,\pi)$, see figures \ref{fig:setup} and \ref{fig:toric_sections}. Notice that the referred torus may be generated by the rotation of a toric section around segment $\overline{SD}$, see figure \ref{fig:toric_sections}.

\begin{figure}
\begin{center}
a\includegraphics[width=0.4\textwidth]{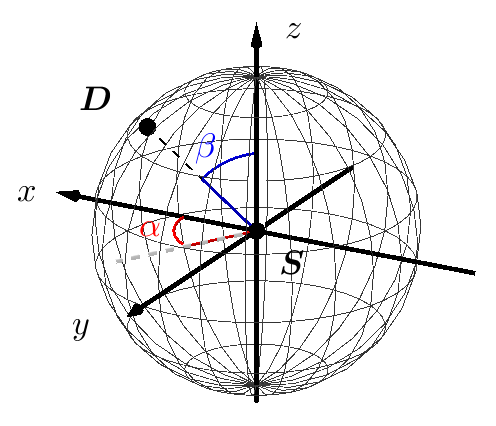}b\includegraphics[width=0.32\textwidth]{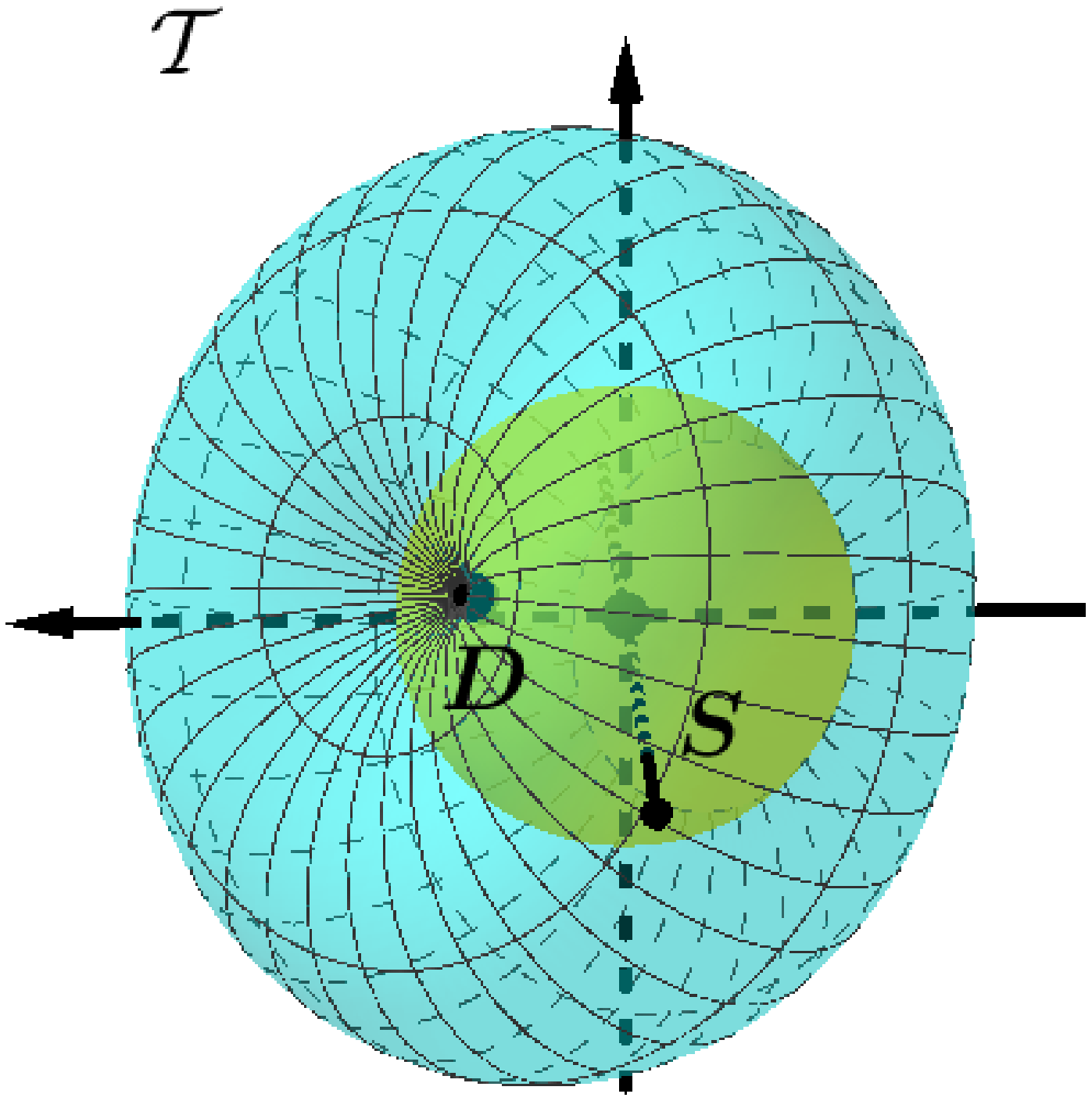}c\includegraphics[width=0.32\textwidth]{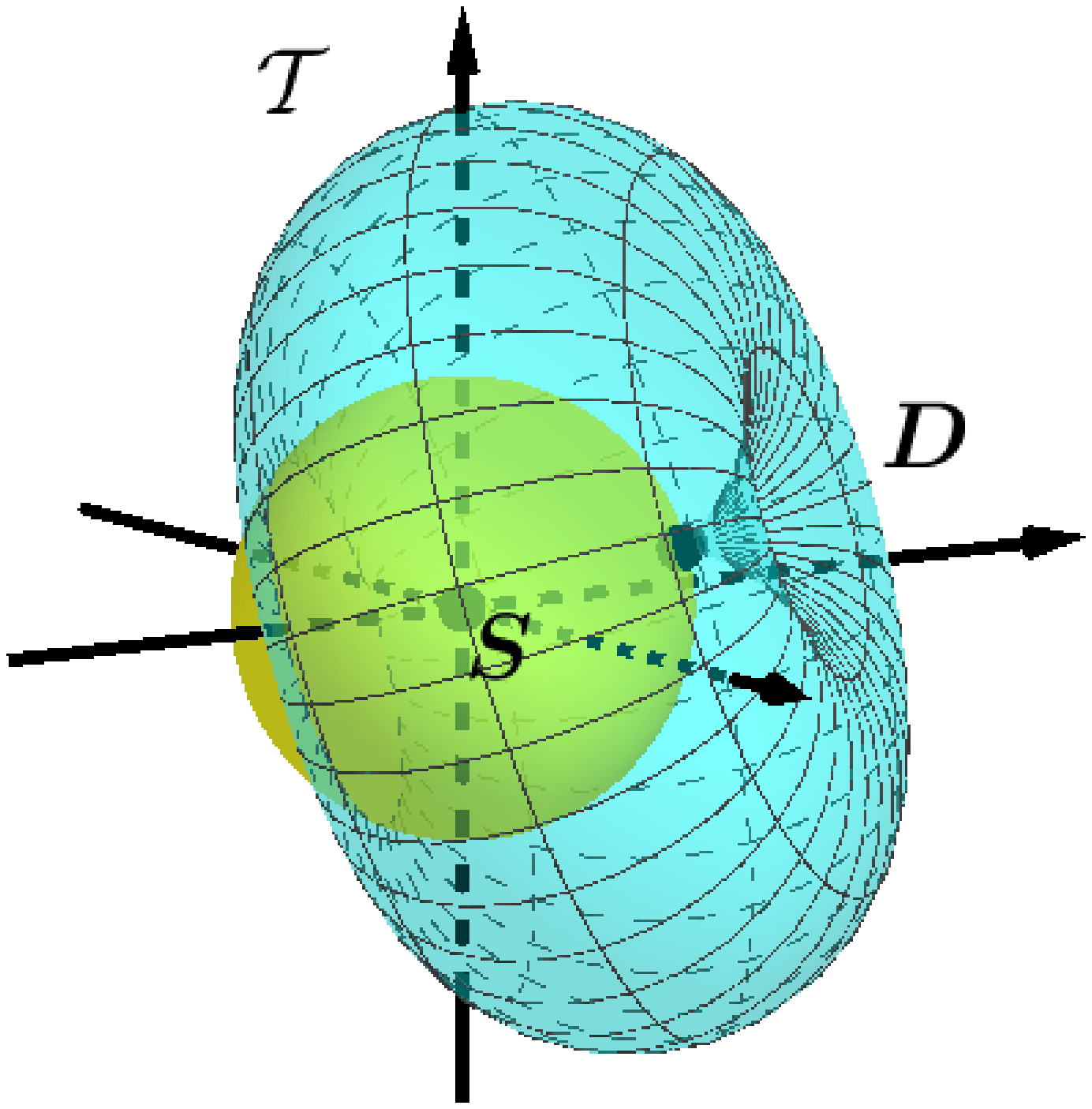}
 \caption{a: setup of the new CST modality: source ($S$), detector ($D$), detection sphere (mesh). b and c: two views of an apple torus $\mathcal{T}$ (blue) and the detection sphere (yellow).} \label{fig:setup}
\end{center}
\end{figure}



\begin{figure}[!h]
\begin{center}
\includegraphics[width= 0.5\textwidth]{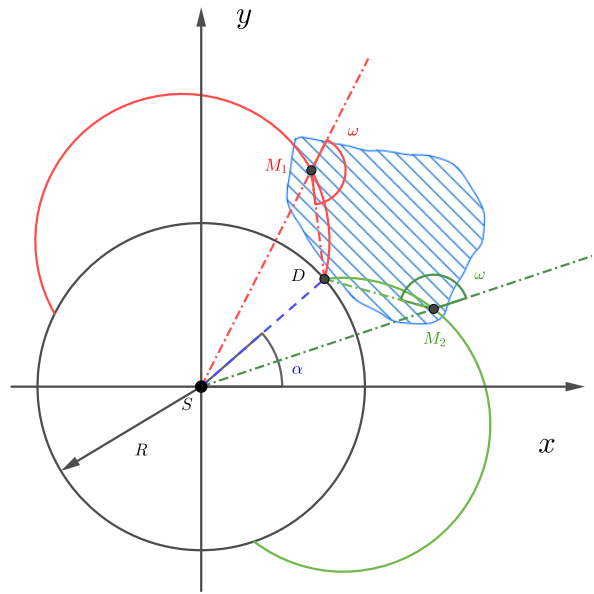}
\caption{Planar section ($z=0$) of the setup and the manifold. Toric section (continous curves red and green) rotates around $\overline{SD}$ to generate the torus. $S$: source, $D$: detector, $\omega$: scattering angle, $M_1$ and $M_2$: scattering sites, $R$: radius of the detection sphere. In this case, the detector is placed at angles $\alpha<\pi/2$ and $\beta=\pi/2$.} \label{fig:toric_sections}
\end{center}
\end{figure}

Before giving the definition of the Radon transform modeling the forward problem, let us introduce some notation. The coordinates of the detector $D$ are given by

\begin{equation*}
D(\alpha,\beta)=R ( \begin{array}{c}\cos\alpha\sin\beta ,
\sin\alpha\sin\beta ,
\cos \beta \end{array} )^T,\quad \alpha \in [0,\,2\pi), \:\beta \in [0,\pi].
\end{equation*}

\noindent
We define a parametrization of any torus $\mathcal{T}^{\omega,\alpha,\beta}$ as:

\begin{equation}\label{eq:phi_omega}
\Phi^{\omega,\alpha,\beta}(\gamma,\psi)= r^\omega(\gamma) \Theta^{\alpha,\beta}(\gamma,\psi),
\end{equation}

\noindent 
with $\gamma \in(0,2\omega - \pi)$, $\psi\in(0,2\pi)$. Here, the radial part is given by

\begin{equation*}
r^\omega(\gamma)=R \frac{\sin ( \omega-\gamma)}{\sin\omega}
\end{equation*}

\noindent
while the angular part is expressed as  $\Theta^{\alpha,\beta}(\gamma,\psi)=u(\alpha)a(\beta)\Theta(\gamma,\psi),$ where
\begin{equation*}
\Theta(\gamma,\psi)= ( \begin{array}{c}\cos\psi\sin\gamma ,
\sin\psi\sin\gamma ,
\cos \gamma \end{array})^T
\end{equation*}
is a point on $S^2$, the unit sphere in $\R^3$, and

\begin{minipage}{0.4\linewidth}
\begin{equation*}\label{eq:u_alpha}
\fl u(\alpha)=\left( \begin{array}{ccc} \cos\alpha & -\sin\alpha & 0 \\
\sin\alpha & \,\,\,\cos\alpha & 0 \\
0 & 0 & 1 \end{array}\right)\,\, \textrm{and}
\end{equation*}
\end{minipage}
\hspace{0.2cm} \begin{minipage}{0.5\linewidth} 
\begin{equation}\label{eq:u_beta}
\fl a(\beta)=\left( \begin{array}{ccc} \cos\beta &  0 &\sin\beta \\
0 & 1 & 0 \\
-\sin\beta & 0 & \cos\beta \\
 \end{array}\right)
\end{equation} 
\end{minipage}
are a rotation of angle $\alpha$ about the $z$-axis and a rotation of angle $\beta$ about the $y$-axis, respectively.
Finally, given positive numbers $r_M,\,r_m$ such that $r_{M}>r_{ m}>R>0$, we define the spherical shell $S_h(r_m,r_M)$ as
$$S_h(r_m,r_M)=\{(x,y,z)\in\R^3:r_m\leq\sqrt{x^2+y^2+z^2}\leq r_M\}.$$

We are now ready to introduce the toric Radon transform associated to this modality.

\bigskip

\noindent{\bf Definition.}
Let $f(x,y,z)$ be a compact supported function with support contained in $S_h(r_m,r_M)$. We define the toric Radon transform $\mathcal{R}_\mathcal{T}f$ of function $f$ as $\mathcal{R}_\mathcal{T}f(\alpha,\beta,\omega)=
\int_\mathcal{T^{\omega,\alpha,\beta}}dS_\mathcal{T} \ f(x,y,z)$. Explicitly

\begin{equation}
\label{eq:TRT_omega}
\mathcal{R}_\mathcal{T}f(\alpha,\beta,\omega)=
\int_0^{2\omega-\pi} d\gamma   \int_0^{2\pi} d\psi \ f \left(  \Phi^{\omega,\alpha,\beta}(\gamma,\psi) \right)  \,r^\omega(\gamma) \frac{\sin \gamma}{\sin \omega},    
\end{equation}

\noindent
where $\omega \in (\frac{\pi}{2},\pi)$, $\alpha\in[0,2\pi)$, $\beta \in [0,\pi]$. 

\bigskip
   
\noindent
Equation~(\ref{eq:TRT_omega}) is the forward operator that models the data recorded, this integral transform is rotational invariant. Its manifold is the part of an apple torus through the origin outside of the sphere of radius $R$. In \cite{WL2018}, the authors studied another rotational invariant toric transform whose manifold was an apple torus that did not pass through the origin. Interesting issues on the toric Radon transform \eref{eq:TRT_omega} such as spherical harmonics expansion \cite{DH1994}, rotational invariance and invertibility \cite{QUINTO1983510} can be studied, in what follows we address some of them. 
%
%
%
%

\subsection{Spherical harmonics expansion}
\label{sec:SH_basics}

In this section, we explicit the connection between the components of the spherical harmonics expansion of a function $f$ and those of its toric Radon transform $\RR_\TT f$. Spherical harmonics of degree $l$ and order $m$ are defined as:

\begin{equation}\label{eq:SH_def}
Y_l^m(\gamma,\psi)=(-1)^m\sqrt{\frac{(2l+1)(l-m)!}{4\pi(l+m)!}}P^m_l(\cos\gamma)e^{im\psi},
\end{equation}

\noindent
where $\gamma \in [0,\pi]$, $\psi\in[0,2\pi)$ and $P^m_l(x)$ is the Legendre polynomial of degree $l$ and order $m$, see \cite{BL81,DH1994} for details. The set $\{Y_l^m\}$, for $l\in\NN$ and $|m|\leq l$ is a complete orthonormal system in $S^2.$ Any function $f \in C^\infty(\R^3)$ can be expanded in terms of $Y_l^m(\gamma,\psi)$ according to

\begin{equation}\label{eq:SH_F}
f(r\Theta(\gamma,\psi))=\sum_{l=0}^\infty\sum_{|m|\leq l}f_{lm}(r)Y_l^m(\gamma,\psi),
\end{equation}

\noindent
where  
$$f_{lm}(r)=\langle f, Y_l^m \rangle=\int_0^{2\pi}\int_0^\pi f(r\Theta(\gamma,\psi))\overline{Y_l^m(\gamma,\psi)}\sin\gamma \,d\gamma d\psi,$$ and $\langle,\rangle$ is the scalar product in $L^2(S^2)$, the overline denotes complex conjugation, and $r\in \R^+_0$ is fixed.

Following the ideas in \cite{WL2018}, it can be shown that the spherical harmonics expansion of $\RR_\TT f$ can be written as

\begin{equation}\label{eq:TRT_SH_0}
\RR_\TT f(\alpha,\beta,\omega)=\sum_{l\in\NN}\sum_{|m|\leq l} (\RR_\TT f)_{lm}(\omega) Y_l^m(\alpha,\beta),
\end{equation}

\noindent
where coefficients in the expansion are given by the following lemma, see \ref{sec:SH_reduction} for a proof.

\begin{lemma} \label{theo:SHE}
The Fourier coefficients of data $(\RR_\TT f)_{lm}$ in \eref{eq:TRT_SH_0} and those of the object $f_{lm}$ are related by

\begin{equation}\label{eq:TRT_SH_RT_lm}
(\RR_\TT f)_{lm}(\omega)=2\pi \int_0^{2\omega-\pi} d\gamma \, r^\omega(\gamma) \frac{\sin \gamma}{\sin \omega} \, f_{lm}  \left(  r^\omega(\gamma) \right) P^{0}_l(\cos\gamma),
\end{equation}
where $P^{0}_l(.)$ is the zero order associated Legendre polynomial of degree $l$.
\end{lemma}

%
%

\section{Uniqueness of the solution of the toric Radon transform}
\label{sec:uniqueness}

We address now the question of the uniqueness of solutions of \eref{eq:TRT_omega} using spherical harmonics expansions. In section \ref{sec:alternative_kernel} we show that \eref{eq:TRT_SH_RT_lm} is a generalized Abel type integral equation with a kernel with zeros on its diagonal. In section \ref{sec:invertibility}  we analyze the properties of the kernel to prove invertibility. 

\subsection{An alternative expression for the integral equation}
\label{sec:alternative_kernel}

We split up \eref{eq:TRT_SH_RT_lm} in two parts having integration  range $(0, {\omega-\frac{\pi}{2}})$  and $({\omega-\frac{\pi}{2}},2\omega-\pi)$, respectively and perform the substitution $\gamma =\omega-\sin^{-1}\left( \frac{r\sin\omega}{R}\right) $ in the first integral and $\gamma =\omega+\sin^{-1}\left( \frac{r\sin\omega}{R}\right) -\pi$ in the second one. After some calculations, making use of the identities $\cos(\sin^{-1}x)=\sqrt{1-x^2}$ and $P_l^0(-x)=(-1)^l P_l^0(x)$ (see, for example, \cite{Laden}), we arrive to

\begin{eqnarray}\label{eq:TRT_SH_to_kernel_3}
\fl (\RR_\TT f)_{lm}(\omega)=
\int_R^{\frac{R}{\sin\omega}} dr \, f_{lm} \left(r \right) \frac{1}{\sqrt{(\frac{R}{\sin \omega})-r}} \times \nonumber\\
\fl \frac{2\pi}{\sin\omega}  {  \sum_{\sigma=\pm 1}} \frac{ r \sin \left( \sin^{-1}\left( \frac{r\sin\omega}{R}\right)-\sigma  \omega \right)}{\sqrt{(\frac{R}{\sin \omega})+r}} (\sigma)^l P^0_l\left( \cos \left(  \omega -\sigma  \sin^{-1}\left( \frac{r\sin\omega}{R}\right)  \right)\right).  \nonumber\\
\quad
\end{eqnarray}

\noindent
The support of function $f_{lm}(r)$ enables to replace the lower integration limit $R$ by $r_{m}$. Making the substitution $p=R/\sin\omega$\footnote{Physically $p$ is the diameter of the circles that generate the torus as a surface of revolution.} and keeping the notation for readability\footnote{Strictly speaking, we should have changed the function after this substitution, for instance $(\RR_\TT f)_{lm}(\sin^{-1}R/p)= \widetilde{(\RR_\TT} f)_{lm}(p)$.}, the integral equation reads:
 
\begin{eqnarray}
\label{eq:TRT_SH_to_kernel_5}
{(\RR_\TT f)_{lm}}(p) =\int_{r_{m}}^{p} dr \, f_{lm} \left(r \right) \frac{1}{\sqrt{p-r}} K_l(p,r),
\end{eqnarray}

\noindent
where the kernel is 

\begin{eqnarray}\label{eq:kernel_1}
\fl
K_l(p,r)=\frac{2\pi}{R}  {  \sum_{\sigma=\pm 1}} \sigma^l  \frac{ p\,r \sin \left( \sin^{-1}\left(\frac{r}{p}\right)-\sigma  \sin^{-1} \left( \frac{R}{p} \right)  \right)}{\sqrt{p+r}} P^0_l\left( \cos \left(   \sin^{-1} \left( \frac{R}{p}\right)  -\sigma  \sin^{-1}\left(\frac{r}{p}\right)   \right)\right). \nonumber\\
\end{eqnarray}
%

\noindent
Now, we rewrite kernel \eref{eq:kernel_1} in a suitable form for the evaluation of the existence of a unique solution of \eref{eq:TRT_SH_RT_lm}. Using elementary geometry, we can express the kernel as

\begin{eqnarray}\label{eq:kernel_3}
\fl
K_l(p,r)=  {  \sum_{\sigma=\pm 1}} \left( Q_1 -\sigma Q_2 \sqrt{p-r} \right) \sigma^l P^0_l\left(  Q_3 \sqrt{p-r} + \sigma Q_4 \right),
\nonumber\\
\quad
\end{eqnarray}

\noindent
where 
 \begin{eqnarray}
Q_1=Q_1(p,r)= \frac{2\pi r^2\sqrt{p^2-R^2}  }{R p \sqrt{p+r}} \\ 
Q_2=Q_2(p,r)= \frac{ 2 \pi r   }{p}\\
Q_3=Q_3(p,r)= \frac{\sqrt{p^2-R^2}\sqrt{p+r} }{p^2} \\
Q_4=Q_4(p,r)=  \frac{ R r}{p^2}.
\end{eqnarray}

\noindent
Using the binomial formula to expand each term of $P^0_l(x)= a_l x^l+...+a_1x+a_0,$ with $x= Q_3 \sqrt{p-r} + \sigma Q_4$, we get 

\begin{eqnarray}\label{eq:kernel_4}
\fl
K_l(p,r)=   {  \sum_{n=0}^l} a_n {  \sum_{k=0}^n} {{n}\choose{k}} Q_3^{n-k}Q_4^k \left[\left( (-1)^{n-k} +1 \right) Q_1 \sqrt{p-r}^{n-k} + \left( (-1)^{n-k} -1 \right)Q_2 \sqrt{p-r}^{n-k+1} \right].
\nonumber\\
\quad
\end{eqnarray}

\noindent
We perform some calculations
\begin{eqnarray*}\label{eq:kernel_3_bis}
\fl
K_l(p,r) = {  \sum_{n=0}^3} a_n {  \sum_{k=0}^n} {{n}\choose{k}} Q_3^{n-k}Q_4^k \left[\left( (-1)^{n-k} +1 \right) Q_1 \sqrt{p-r}^{n-k} + \left( (-1)^{n-k} -1 \right) Q_2 \sqrt{p-r}^{n-k+1} \right]\cr
 \fl\phantom{doremi}+  \sum_{n=4}^l a_n \left\{ {  \sum_{k=0}^{n-4}} {{n}\choose{k}} Q_3^{n-k}Q_4^k \left[\left( (-1)^{n-k} +1 \right) Q_1 \sqrt{p-r}^{n-k} + \left( (-1)^{n-k} -1 \right) Q_2 \sqrt{p-r}^{n-k+1} \right] \right. \cr
\fl\phantom{doremi} \left. +{  \sum_{k=n-3}^{n}} {{n}\choose{k}} Q_3^{n-k}Q_4^k \left[\left( (-1)^{n-k} +1 \right) Q_1 \sqrt{p-r}^{n-k} + \left( (-1)^{n-k} -1 \right) Q_2 \sqrt{p-r}^{n-k+1} \right]\right\} \cr
\fl\phantom{doremi}= \sum_{n=4}^l a_n  {  \sum_{k=0}^{n-4}} {{n}\choose{k}} Q_3^{n-k}Q_4^k \left[\left( (-1)^{n-k} +1 \right) Q_1 \sqrt{p-r}^{n-k} + \left( (-1)^{n-k} -1 \right) Q_2 \sqrt{p-r}^{n-k+1} \right]\cr
\fl\phantom{doremi} +(p-r)^2\left[-2a_3{{3}\choose{0}}Q_3^3\,Q_2-2\sum_{n=4}^l a_n {{n}\choose{n-3}}Q_3^3\,Q_4^{n-3}\,Q_2\right]\cr
\fl\phantom{doremi}+2(p-r) \left\{  \left[Q_3^2\,Q_1\sum_{n=4}^l a_n {{n}\choose{n-2}}Q_4^{n-2}-Q_3\,Q_2\sum_{n=4}^l a_n {{n}\choose{n-1}}Q_4^{n-1}\right] \right. \cr
\fl\phantom{doremi}\left. +Q_3^2\,Q_1\left[a_2{{2}\choose{0}}+a_3{{3}\choose{1}}Q_4\right]-Q_3\,Q_2\left[a_1{{1}\choose{0}}+a_2{{2}\choose{1}}Q_4+a_3{{3}\choose{2}}Q_4^2\right] \right\} \cr
\fl\phantom{doremi}+2Q_1\left[\sum_{n=4}^l a_n {{n}\choose{n}}Q_4^n+a_3{{3}\choose{3}}Q_4^3+a_2{{2}\choose{2}}Q_4^2+a_1{{1}\choose{1}}Q_4+a_0{{0}\choose{0}}\right],\cr
\
\end{eqnarray*}
and finally we arrive to

\begin{eqnarray}\label{eq:kernel_5}
\fl
K_l(p,r)= 2 Q_1   {\sum_{n=0}^l} a_n  Q_4^n + 2(p-r)\left[ \frac{1}{2} Q_3^2Q_1  {\sum_{n=2}^l} a_n n(n-1) Q_4^{n-2}  -Q_3Q_2  {\sum_{n=1}^l} a_n n \, Q_4^{n-1}   \right]\cr
\fl + {  \sum_{n=3}^l} a_n {  \sum_{k=0}^{n-3}} {{n}\choose{k}} Q_3^{n-k}Q_4^k \left[\left( (-1)^{n-k} +1 \right) Q_1 \sqrt{p-r}^{n-k} + \left( (-1)^{n-k} -1 \right) Q_2 \sqrt{p-r}^{n-k+1} \right].\cr
\nonumber\\
\quad
\end{eqnarray}

\noindent
Notice that summations on the first line correspond to Legendre polynomials and their first and second order derivatives evaluated in $Q_4$. Also, terms with odd powers of $\sqrt{p-r}$ vanish and only powers of $(p-r)$ greater than or equal to 2 remain.  Then, using a compact notation for Legendre polynomials $P(.)=P^0_l(.)$ and its derivatives $P'(.)$ and $P''(.)$, and splitting up summations in order to define a new subindex, we have

\begin{eqnarray}\label{eq:kernel_7}
\fl K_l(p,r)=2 Q_1   P(Q_4) + 2(p-r)\left[ \case{1}{2} Q_3^2Q_1  P''(Q_4)  -Q_3Q_2  P'(Q_4) \right]\cr
\fl + 2 {  \sum_{n=3}^l} a_n \left[  Q_1 {  \sum_{i=2}^{\lfloor \case{n}{2} \rfloor}} {{n}\choose{n-2i}} Q_3^{2i}Q_4^{n-2i} (p-r)^i -Q_2 {  \sum_{i=2}^{\lfloor \case{n+1}{2} \rfloor}} {{n}\choose{n-2i+1}} Q_3^{2i-1}Q_4^{n-2i+1} (p-r)^i \right] .\cr\nonumber\\
\end{eqnarray}

\noindent 
Equation \eref{eq:kernel_7} is useful for studying the properties of the kernel in order to prove invertibility. Particularly, \eref{eq:kernel_7} is more suitable for differentiation on the diagonal $p=r$ according Theorem \ref{theo:theo_Haltmeier} in next section.

\subsection{On the invertibility of $\RR_\TT f$ in terms of spherical harmonics}
\label{sec:invertibility}

Equation \eref{eq:TRT_SH_to_kernel_5} is a generalized Abel type equation. Notice that kernel \eref{eq:kernel_7} has zeros on the diagonal $p=r$, since the Legendre polynomials have zeros in the interval $[0, 1]$. Recently, Schiefeneder and Haltmeier stated conditions for the uniqueness of the solution of this kind of equations with kernels with zeros on its diagonal \cite{DSMH2017}. For the sake of completeness, we present their results in the following lemma.



\begin{lemma}[\cite{DSMH2017}, Theorem 3.4]\label{theo:theo_Haltmeier}
Consider the generalized Abel type integral equation 

\begin{equation}\label{eq:abel}
\forall t\in[a,b]:g(t)=\int_a^tds \,f(s) \frac{1}{\sqrt{t-s}}K(t,s)
\end{equation}
\noindent
where $g \in C([a,b])$  and $K \in C(\bigtriangleup(a,b))$, with $\bigtriangleup(a,b):=\lbrace a \leq s \leq t \leq b \rbrace$, is a continuous kernel having zeros on the diagonal. Suppose $K:\bigtriangleup(a,b)\rightarrow \R$, where $a<b$, satisfies the following assumptions:

\begin{enumerate} 
\item $K\in C^3(\bigtriangleup(a,b)).$
\item $N_K:=\lbrace s\in[a,b)|K(s,s)=0\rbrace$ is finite and consists of simple roots.
\item For every $s\in N_K$, the gradient $(\kappa_1,\kappa_2)=\nabla K(s,s)$ satisfies
$$1+\frac{1}{2}\frac{\kappa_1}{\kappa_1+\kappa_2}>0.$$
\end{enumerate}
Then, for any $g\in C([a,b])$, equation \eref{eq:abel} has at most one solution $f\in C([a,b])$.
\end{lemma} 

We are now ready to introduce our claim.
\begin{theorem}[Invertibility of equation \eref{eq:TRT_SH_to_kernel_5}] \label{theo:main}
For any $(\RR_\TT f)_{lm}\in C([r_m,r_M])$, equation \eref{eq:TRT_SH_to_kernel_5} has at most one solution $f_{lm} \in C([r_m,r_M])$.
\end{theorem}

\begin{proof}
In order to show the uniqueness of the solution of \eref{eq:TRT_SH_to_kernel_5}, we use lemma \ref{theo:theo_Haltmeier}. So, we must check assumptions $(i)$, $(ii)$ and $(iii)$. We define the triangle $\bigtriangleup(r_{m},r_{M}):=\lbrace r_{m} \leq r \leq p \leq r_{M} \rbrace$. 
\begin{enumerate} 
\item 
As it is evident from \eref{eq:kernel_3}, the kernel is smooth inside $\bigtriangleup(r_{m},r_{M})$. Notice in \eref{eq:kernel_4} that terms containing odd powers of $\sqrt{p-r}$ always vanish. Moreover, the double summation involves only non-negative integer powers of $p-r$ times products of powers of factors $Q_1$, $Q_2$, $Q_3$ and $Q_4$. Since all the functions involved are smooth on the diagonal $p=r$, so is $K_l$.

\item 
From \eref{eq:kernel_7} we obtain the expression for the kernel on the diagonal

$$K_l(r,r)=\case{\sqrt{8}\pi}{R}{\sqrt r}\sqrt{r^2-R^2}\,P\left( \frac{R}{r}\right).$$

So, $K_l(r,r)$ has the same zeros as $P\left( \case{R}{r}\right)$ since $0<R<r_{m}\leq r$. Given that $P$ is an orthogonal polynomial, $P\left( \case{R}{r}\right)$ has a finite number of simple roots. Then, the values $r_0$ such that $P\left( \case{R}{r_0}\right)=0$ are the zeros of the kernel on the diagonal, i.e. $K_l(r_0,r_0)=0$.
 
\item 
We obtain an expression for $\nabla K_l(r,r)$ and evaluate it at zeros $r_0$. We see in \eref{eq:kernel_7} that the derivatives of terms containing factors $(p-r)^i$ with $i=2,3,...$ vanish on the diagonal $p=r$. So, we define a new function

\begin{eqnarray}\label{eq:kernel_ii}
\fl
\overline K_l(p,r)= 2 Q_1   P(Q_4) + 2(p-r)\left[ \case{1}{2} Q_3^2Q_1  P''(Q_4)  -Q_3Q_2  P'(Q_4) \right],
\nonumber\\
\quad
\end{eqnarray}

\noindent
which verifies $\nabla K_l(r,r)=\nabla \overline K_l(r,r)$ and is easier to handle. Then, we differentiate $\overline K_l$ with respect to $p$

\begin{eqnarray}\label{eq:kernel_dp}
\fl
\frac{\partial \overline  K_l}{\partial p}(p,r)= 2 \frac{\partial Q_1 }{\partial p} P(Q_4)
+  2  Q_1 \frac{\partial }{\partial p} P(Q_4)
+ 2 \left[ \case{1}{2} Q_3^2Q_1  P''(Q_4) -Q_3Q_2  P'(Q_4)  \right]\cr 
+2(p-r) \frac{\partial }{\partial p} \left[ \case{1}{2} Q_3^2Q_1  P''(Q_4)  -Q_3Q_2  P'(Q_4)   \right] .
\nonumber\\
\quad
\end{eqnarray}

\noindent
Using the identity $P''(x)=[2xP'(x)-l(l+1)P(x)]/(1-x^2),$ see for example \cite{Laden,Jackson}, and the fact that $Q_4(p,r)<1,$ we substitute $P''(Q_4)$ in the third term in \eref{eq:kernel_dp}, and we arrive to

\begin{eqnarray}\label{eq:kernel_dp2}
\fl
\frac{\partial \overline  K_l}{\partial p}(p,r)= \left(  2 \frac{\partial Q_1 }{\partial p} - \case{Q_3^2Q_1 l(l+1)}{1-Q_4^2}\right) P(Q_4)
+2(p-r) \frac{\partial }{\partial p} \left[\case{1}{2} Q_3^2Q_1  P''(Q_4)  -Q_3Q_2  P'(Q_4)   \right] \cr
+ 2 \left[  Q_1\frac{\partial Q_4 }{\partial p} - Q_3Q_2 + \case{Q_3^2Q_1Q_4}{1-Q_4^2}  \right]P'(Q_4).
\nonumber\\
\quad
\end{eqnarray}

\noindent
We proceed identically for the other variable $r$

\begin{eqnarray}\label{eq:kernel_dr}
\fl
\frac{\partial \overline  K_l}{\partial r}(p,r)= 2 \frac{\partial Q_1 }{\partial r} P(Q_4)+  2  Q_1 \frac{\partial }{\partial r} P(Q_4)
- 2 \left[ Q_3^2Q_1  \case{1}{2} P''(Q_4)  -Q_3Q_2 P'(Q_4)   \right]\cr 
+2(p-r) \frac{\partial }{\partial r} \left[ \case{1}{2} Q_3^2Q_1  P''(Q_4)  - Q_3Q_2  P'(Q_4)   \right] 
\nonumber\\
\quad
\end{eqnarray} 

\noindent
and we obtain

\begin{eqnarray}\label{eq:kernel_dr2}
\fl
\frac{\partial \overline  K_l}{\partial r}(p,r)= \left(  2 \frac{\partial Q_1 }{\partial r} + \case{Q_3^2Q_1 l(l+1)}{1-Q_4^2}\right) P(Q_4)
+2(p-r) \frac{\partial }{\partial r} \left[\case{1}{2} Q_3^2Q_1  P''(Q_4)  -Q_3Q_2  P'(Q_4)   \right] \cr
+ 2 \left[  Q_1\frac{\partial Q_4 }{\partial r} + Q_3Q_2 - \case{Q_3^2Q_1Q_4}{1-Q_4^2}  \right]P'(Q_4).
\nonumber\\
\quad
\end{eqnarray} 

\noindent
The explicit forms \eref{eq:kernel_dp2} and \eref{eq:kernel_dr2} simplify the calculation of $\frac{\partial K_l}{\partial r}(r,r)$ and $\frac{\partial K_l}{\partial p}(r,r)$. Evaluating at $p=r=r_0,$ the first and second terms in \eref{eq:kernel_dp2} and \eref{eq:kernel_dr2} vanish. In fact, the first terms vanish because $P(Q_4(r_0,r_0))=P(\frac{Rr_0}{r_0^2})=P(\frac{R}{r_0})=0,$ while the second terms vanish because of the factor $(p-r).$ Thus, we have

\begin{eqnarray}\label{eq:K_grad}
\fl
(\kappa_1,\kappa_2)=\nabla K_l(r_0,r_0)=\nabla \overline K_l(r_0,r_0)=\cr
\fl 2\left(   \left[  Q_1\frac{\partial Q_4 }{\partial p} - Q_3Q_2 + \case{Q_3^2Q_1Q_4}{1-Q_4^2}  \right] ,
 \left[  Q_1\frac{\partial Q_4 }{\partial r} + Q_3Q_2 - \case{Q_3^2Q_1Q_4}{1-Q_4^2}  \right]\right)P'(Q_4)  \biggr|_{p=r=r_0}.
\nonumber\\
\quad
\end{eqnarray} 

\noindent
We evaluate (\ref{eq:K_grad}) using the definition of functions $Q_1,\,Q_2,\,Q_3$ and $Q_4$, then the expressions for the gradient components are

\begin{equation}
\kappa_1= - 4 \pi P'\left( \case{R}{r_0} \right) \case{\sqrt{2 r_0}\sqrt{r_0^2-R^2}}{r_0^2}
\end{equation}

\noindent and

\begin{equation}
\kappa_2= 2 \pi P'\left( \case{R}{r_0} \right) \case{\sqrt{2 r_0}\sqrt{r_0^2-R^2}}{r_0^2}.
\end{equation}

\noindent 
Finally, we get the ratio

$$ 1 + \frac{1}{2}\frac{\kappa_1}{\kappa_1+\kappa_2}=2>0,$$

\noindent as we wanted to show.

\end{enumerate}

\noindent
Thus, the radial components $f_{lm}(r)$ of function $f$ can be recovered uniquely from the coefficients $(\RR_\TT f)_{lm}(p)$.

\end{proof}

\noindent
We can express data in terms of variable $p$ instead of the scattering angle $\omega$

$$\RR_\TT f(p,\alpha,\beta) = \sum_{l\in\NN}\sum_{|m|\leq l}(\RR_\TT f)_{lm}(p) Y_l^m(\beta,\alpha).$$
\noindent
From the geometrical point of view, $p$ is the diameter of the circles generating the torus, see section \ref{sec:direct_model} for the integral definition of $\RR_\TT f(p,\alpha,\beta)$.
The following corollary summarizes our result on uniqueness.
%

%
%
%

\begin{corollary}[Invertibility of the  $\RR_\TT f$]\label{theo:coro} If $f_1$ and $f_2$ are compact supported functions in $C^\infty(S_h(r_m,r_M))$ and  $\RR_\TT f_1=\RR_\TT f_2$, then $f_1=f_2$.


\end{corollary}

\begin{proof}
Let $f$ satisfy $(\RR_\TT f)_{lm}=0$ for all $l,m$ in the spherical harmonics expansion. According to Theorem \ref{theo:main}, there is a unique solution $f_{lm}=0$, which implies $f=0$. The linearity of  $\RR_\TT f$ yields the claim.
\end{proof}

\section{Numerical Simulations}

\label{sec:NS}

We use a discrete spherical harmonics expansion of the functions representing the data and the object in order to carry out numerical reconstructions. We employ a product integration approach to model the discrete problem in the spherical harmonic domain. Then, Tikhonov regularization is used to solve a set of normal equations.
\subsection{Alternative forward model}

\label{sec:direct_model}
An alternative definition of our toric Radon transform useful in numerical simulation is

\begin{equation}\label{eq:TRT_p2}
\fl \RR_\TT f(p,\alpha,\beta)=p^2
\int_0^\pi d\gamma\int_0^{2\pi}d\psi  \cos \left(  \gamma-\cos^{-1}\frac{R}{p} \right)   \sin\gamma f\left(   \Phi^{p,\alpha,\beta}(\gamma,\psi) \right),      
\end{equation}

\noindent
where $p\in(R,+\infty)$ is the diameter of the circles making the torus and 

\begin{equation}\label{eq:phi_p}
\Phi^{p,\alpha,\beta}(\gamma,\psi)=r \Theta^{\alpha,\beta}(\gamma,\psi)|_{r=p\cos(\gamma-\cos^{-1}\frac{R}{p})}
\end{equation}
\noindent
is the parametrization of the toric surface labeled by variables  $(p,\alpha,\beta)$. The unit vector $\Theta^{\alpha,\beta}(\gamma,\psi)$ was introduced in the definition \eref{eq:phi_omega}. The scattering angle $\omega$ and the diameter $p$ are related through $p=R/\sin\omega$.


\subsection{The algebraic problem}

In this section we use a discrete spherical harmonics expansion of order $N$ to write the problem as an algebraic product suitable for Tikhonov regularization. This is achieved using numerical algorithms for the Discrete-Inverse Spherical Harmonics transform (DSHT-IDSHT), see an outline in  \ref{sec:DSHE}, with data as well as with the sought function according to equations \eref{eq:g_DSHT} and \eref{eq:f_DSHT}. The pair DSHT-IDSHT allows us to write the problem in the domain of the spherical harmonics according to

\begin{equation}\label{eq:g_DSHT}
\mathbf{g}_{nk}^j  \mathop{\rightleftarrows}^{\mathrm{DSHT}}_{\mathrm{IDSHT}} \mathbf{g}_{lm}^j
\end{equation}

\begin{equation}\label{eq:f_DSHT}
 \mathbf{f}_{nk}^i  \mathop{\rightleftarrows}^{\mathrm{DSHT}}_{\mathrm{IDSHT}} \mathbf{ f }_{lm}^i
\end{equation}

\noindent 
where the discrete functions are $\mathbf{g}_{nk}^{j}=  \RR_\TT f(p_j,\alpha_n,\beta_k)$, $\mathbf{g}_{lm}
^j=\left( \RR_\TT f \right)_{lm} (p_j)$ with $j=0,...,N_p-1$, $\mathbf{f}_{nk}^i=f(r_i\cos\psi_n\sin\gamma_k, r_i\sin\psi_n\sin\gamma_k,r_i\cos \gamma_k)$ and $\mathbf{f}_{lm}^i=f_{lm}(r_i)$ with $i=0,...,N_r-1$. Variables $p$ and $r$ are in the range $(R,r_M^*]$ where $r_M^*$ is such that $r_M^*\geq r_M$ to cover the radial support of function $f$. In this domain, the components of the vector representing the unknown function $\mathbf{f}_{lm}$ are related to components of the vector for the known data $\mathbf{g}_{lm}$ through the equation $\mathbf{g}_{lm}=A_l \mathbf{f}_{lm}$ that is an algebraic relative of equation \eref{eq:TRT_SH_to_kernel_5}. We are aimed to solve the equation for each combination $l,m$.

\subsection{Matrix generation}
\label{sec:matrix}

Matrix $A_l \in \R^{N_p \times N_r}$ is the key for solving the numerical inverse problem. Given that we use an expansion of order $N$ and the kernel in \eref{eq:TRT_SH_to_kernel_5} is $l$-dependent and $m$-independent, there are only $N+1$ different matrices. Adopting the convention $Np=Nr=M$ and splitting up the integration range in equation \eref{eq:TRT_SH_to_kernel_5}, we rewrite

\begin{equation}
\left( \RR_\TT f \right)_{lm} (p_j) = \sum_{q=1}^{j} \int_{r_{q-1}}^{r_{q}} dr \, f_{lm}(r) \frac{r}{\sqrt{p_j^2-r^2}} \tilde {K_l}(p_j,r)
\end{equation}

\noindent
where $r_{q}=R+q(r_{M}^*-R)/M$ and $\tilde {K_l}(p_j,r)=\sqrt{p_j + r}{K_l}(p_j,r)/r$. We use product integration \cite{haltmeier2017inversion} but instead of using the mid-point rule we approximate $\tilde {K_l}(p_j,r)$ by its average  $\tilde {K_l}_{q}(p_j)$  in ten equidistant points in the interval $[r_{q-1},r_{q}]$. Thus, the discrete form for the equation is 


\begin{equation}\mathbf{g}_{lm}^j= \left( \RR_\TT f \right)_{lm} (p_j)\simeq \sum_{q=1}^{j} w_{j,q}\tilde {K_l}_{q}(p_j)f_{lm}(r_q),\end{equation}

\noindent
where the weighting factor has been calculated analytically according to

\begin{equation}w_{j,q} : = \int_{r_{q-1}}^{r_{q}} dr  \frac{r}{\sqrt{p_j^2-r^2}}
\end{equation}

\noindent
with $j,q=1,...,M$ and $w_{j,q}=0$ if $j<q$. Finally, the entries of the lower-triangular matrix in equation $\mathbf{g}_{lm}=A_l \mathbf{f}_{lm}$ are

\begin{equation}\label{eq:matrix_A}
A_l = \left( w_{j,q} \tilde {K_l}_{q}(p_j) \right)_{j,q=1,....M} \in \R^{M \times M}.
\end{equation}

\subsection{Overview of the reconstruction algorithm}

The aim now is recovering vector $\mathbf{f}_{lm}$ from $\mathbf{g}_{lm}$ using the equation $\mathbf{g}_{lm}=A_l \mathbf{f}_{lm}$. When $A_l$ is non-singular and well-conditioned the problem can be easily solved by forward substitution. Because the kernel has zeros on its diagonal, matrix $A_l$ may have diagonal entries being zero or close to zero. Thus, solving the system may be ill-conditioned and regularization methods must be applied. The algorithm is summarized in Table \ref{tab:algo_1} and is aimed to solve the matrix problem $\mathbf{g}_{lm}=A_l \mathbf{f}_{lm}$ for $l=0,...,N$ and $|m| \leq l$. Tikhonov regularization requires to solve the normal equations

\begin{equation}\label{eq:normal_eq}
 (A_{l}^T A_{l}+ \lambda I)\mathbf{f}_{lm} = A_{l}^T\mathbf{g}_{lm},
\end{equation}

\noindent
where $I$ is the identity matrix and $\lambda$ is a regularization parameter. 
 
\begin{table}

\caption{\label{tab:algo_1}Sketch of the reconstruction algorithm.}
\begin{indented}
\item[]\begin{tabular}{l}
\br
\rm Matrix $A_l$ is precalculated according to \eref{eq:matrix_A}.\\

\mr
1:  Chose a suitable value for $\lambda$ according to signal conditions\\
2:  Perform DSHT \eref{eq:g_DSHT} to data $\mathbf{g}_{nk}^j$ to obtain $\mathbf{g}_{lm}^j $\\
3:  For each pair $l,m$ : solve the normal equations in \eref{eq:normal_eq}\\
    \quad and obtain an approximation of $\mathbf{f}_{lm}^i$\\ 
4:  Perform IDSHT \eref{eq:f_DSHT} to obtain the reconstruction $\mathbf{f}_{nk}^i$ \\
5:  Interpolate to obtain the function in discrete Cartesian coordinates  $\tilde \mathbf{f}$\\
\br
\end{tabular}
\end{indented}
\end{table}

For non-vanishing kernel diagonals, the product integration method is convergent \cite{Weiss71}. To the best of our knowledge, there are no reported results on the convergence of product integration with vanishing kernel-diagonals. As suggested in \cite{haltmeier2017inversion}, numerical evidence indicate that, for a suitable selection of the regularization parameter, a convergence analysis may be possible.

\subsection{Results}

\label{sec:results}


Data was simulated using equation \eref{eq:TRT_p2}, a system where detector moves on a sphere of radius $R=1/8$ was considered. The object was a $64\times64\times64$ volume with two balls with different intensities, and different contrasts, gray on black, white on gray, etc. One object has also a defect (crack) in some planes, see figure \ref{fig:phantom}. The function was supported in a cube of side $L=1$ in the first octant with coordinate $(x_{\min},y_{\min},z_{\min})=(L/64,L/64,R)$. Discretization parameters of data are: $N_\alpha=513$, $N_\beta=256$ and $N_p=512$. The maximal diameter of the circles generating the torus is $r^*_{M}=2r_{M}$ where $[R, \,r^*_{M}]$ is the radial support of the cube containing the phantom. Numerical integration is performed in variables $\gamma$ and $\psi$ with $\Delta \psi = 2\pi/N_\psi$ and the variable step  $\Delta \gamma =  \cos^{-1}  \left( p/R \right)    /N_\gamma$, with $N_\gamma=256$ and $N_\psi=256$. Figure 
\ref{fig:data_1} shows simulated data for different values of angle $\alpha$. The trapezoidal rule was used to perform numerical integration. According to the discretization chosen for data, the order of the spherical harmonics expansion was $N=256$ ($N_\alpha=2N+1$). In order to get more realistic simulations, data  $\mathbf{g}=\mathbf{g}_{nk}^j$ were also corrupted with additive Gaussian noise with zero mean. Noise variance was manually adjusted to get several signal-to-noise ratios: SNR$=10$, $20$ and $30$ dB. These SNR correspond to $\epsilon=29$, $10$ and $3\%$ noise levels where $\epsilon=100||\mathbf{\tilde g}-\mathbf{ g}||_2/|| \mathbf{ g}||_2$, and $\mathbf{\tilde g}$ is the corrupted data. In order to asses the quality of reconstruction we used the following measures of error: the Normalized Mean Square Error ($\%$)

\begin{equation}\label{eq:NMSE}
\mathrm{NMSE}=\frac{100}{N^3}\frac{ ||\mathbf{f} -\tilde {\mathbf{ f}}||_2^2 }{\max_i \left\lbrace \mathbf{f}_i^2 \right\rbrace }
\end{equation}

\noindent
and the Normalized Mean Absolute Error ($\%$)

\begin{equation}\label{eq:NASE}
\mathrm{NMAE}=\frac{100}{N^3}\frac{|| \mathbf{f} -\tilde {\mathbf{ f}} ||_1 }{\max_i \left\lbrace \mathbf{f}_i \right\rbrace },
\end{equation}

\begin{figure}
\begin{center}
\begin{tabular}{ccccc}

\resizebox*{3.5cm}{!}{\includegraphics{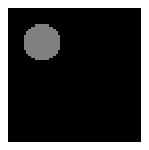}} & \resizebox*{3.5cm}{!}{\includegraphics{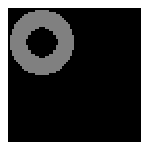}} &
\resizebox*{3.5cm}{!}{\includegraphics{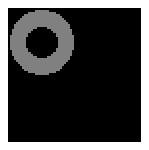}}& \resizebox*{3.5cm}{!}{\includegraphics{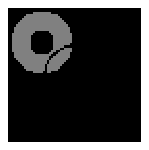}}& \resizebox*{0.8cm}{!}{\includegraphics{barra.eps}}\\

\resizebox*{3.5cm}{!}{\includegraphics{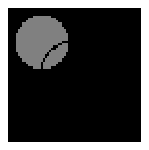}} & \resizebox*{3.5cm}{!}{\includegraphics{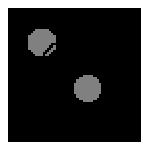}} &
\resizebox*{3.5cm}{!}{\includegraphics{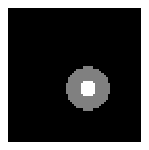}}& \resizebox*{3.5cm}{!}{\includegraphics{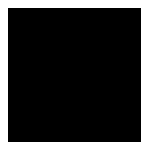}}& \resizebox*{0.8cm}{!}{\includegraphics{barra.eps}}\\

\end{tabular}
\end{center}
\caption{\label{fig:phantom} Original 3D phantom used for simulations. Notice the crack in some planes.}
\end{figure}
\begin{figure}
\begin{center}
\includegraphics[width=0.35\textwidth]{g1.eps}\includegraphics[width=0.35\textwidth]{g192.eps}\\\includegraphics[width=0.35\textwidth]{g320.eps}\includegraphics[width=0.35\textwidth]{g384.eps}

\caption{\label{fig:data_1} Data simulated using \eref{eq:TRT_p2}. Function  $\RR_\TT(p,\alpha,\beta)$ is shown for the sets $\alpha=0,\,3\pi/4,\,5\pi/4,\,3\pi/2$.} 
\end{center}
\end{figure}

\noindent
where  $\mathbf{f}$ is the original image and $\tilde{\mathbf{ f}}$ is the reconstruction. We used the algorithm in Table \ref{tab:algo_1} to carry out reconstructions for two values of the regularization parameter, $\lambda=0.01$ (noiseless data) and $\lambda=0.05$ (noisy). These values were chosen heuristically and are the same for all equations in a set \eref{eq:normal_eq}, i.e. they do not depend on $l,m$. Figure \ref{fig:noiseless} shows reconstructions for noiseless data and figures \ref{fig:30dB} to \ref{fig:10dB} show reconstructions from corrupted data. No post-processing was applied to images. The original function as well as reconstructions are shown at planes $z= 4$, 14, 15, 22, 26, 31, 38 and 58 from top to bottom and from left to right and error metrics are displayed in figure captions. Reconstructions exhibit acceptable quality: structures inside the object are distinguished, shapes are kept and error metrics seem to be reasonable. Moreover, the crack is visible in all the images where its is expected. There are, however, background artifacts and blurring, particularly in upper planes. As expected, stronger regularization is required for noisy data and error metrics get worse with the level of noise. Although Tikhonov regularization performs well, reconstructions can be improved using more advanced regularization techniques. There are algorithms enforcing total variation minimization in the Cartesian domain \cite{Condat13,Condat14} that can be applied after step 5 in the algorithm in Table \ref{tab:algo_1}.  

From an algorithmic point of view, the reconstruction approach allows to split up the problem in several equations reducing the size of the matrices involved with respect to a standard algebraic treatment. In this framework, computation time may be saved through parallelization since the resulting algebraic equations are independent.

%
%

%

\begin{figure}
\begin{center}
\begin{tabular}{ccccc}

\resizebox*{3.5cm}{!}{\includegraphics{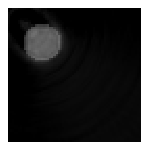}} & \resizebox*{3.5cm}{!}{\includegraphics{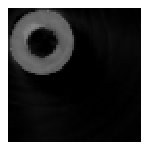}} &
\resizebox*{3.5cm}{!}{\includegraphics{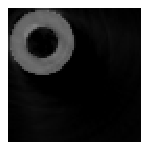}}& \resizebox*{3.5cm}{!}{\includegraphics{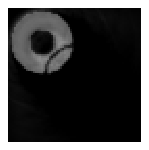}}& \resizebox*{0.8cm}{!}{\includegraphics{barra.eps}}\\

\resizebox*{3.5cm}{!}{\includegraphics{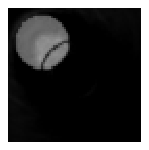}} & \resizebox*{3.5cm}{!}{\includegraphics{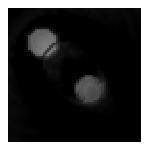}} &
\resizebox*{3.5cm}{!}{\includegraphics{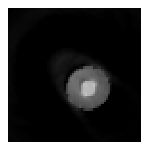}}& \resizebox*{3.5cm}{!}{\includegraphics{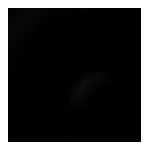}}& \resizebox*{0.8cm}{!}{\includegraphics{barra.eps}}\\

\end{tabular}
\end{center}
\caption{\label{fig:noiseless} Reconstructions from noiseless data. Errors are NMSE=0.32 $\%$, NMAE=3.81 $\%$. Regularization parameter $\lambda=0.01$.} 
\end{figure}

%
%
%
%
%
%
%

\begin{figure}
\begin{center}
\begin{tabular}{ccccc}

\resizebox*{3.5cm}{!}{\includegraphics{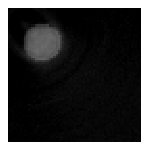}} & \resizebox*{3.5cm}{!}{\includegraphics{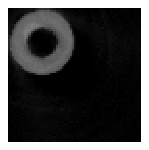}} &
\resizebox*{3.5cm}{!}{\includegraphics{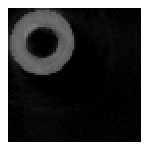}}& \resizebox*{3.5cm}{!}{\includegraphics{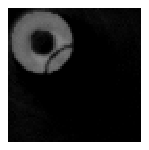}}& \resizebox*{0.8cm}{!}{\includegraphics{barra.eps}}\\

\resizebox*{3.5cm}{!}{\includegraphics{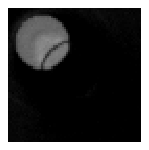}} & \resizebox*{3.5cm}{!}{\includegraphics{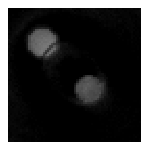}} &
\resizebox*{3.5cm}{!}{\includegraphics{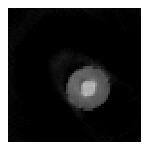}}& \resizebox*{3.5cm}{!}{\includegraphics{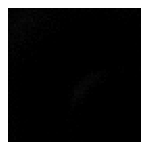}}& \resizebox*{0.8cm}{!}{\includegraphics{barra.eps}}\\
\end{tabular}
\end{center}
\caption{\label{fig:30dB} Reconstructions from noisy data with SNR=30 dB (relative level of noise 3 $\%$). Errors are NMSE=0.34 $\%$, NMAE=3.77 $\%$. Regularization parameter $\lambda=0.05$.} 
\end{figure}

\begin{figure}
\begin{center}
\begin{tabular}{ccccc}

\resizebox*{3.5cm}{!}{\includegraphics{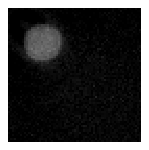}} & \resizebox*{3.5cm}{!}{\includegraphics{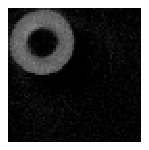}} &
\resizebox*{3.5cm}{!}{\includegraphics{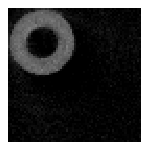}}& \resizebox*{3.5cm}{!}{\includegraphics{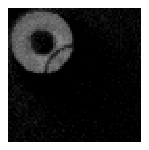}}& \resizebox*{0.8cm}{!}{\includegraphics{barra.eps}}\\

\resizebox*{3.5cm}{!}{\includegraphics{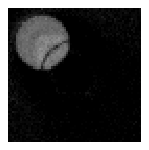}} & \resizebox*{3.5cm}{!}{\includegraphics{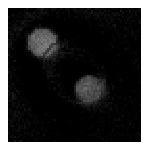}} &
\resizebox*{3.5cm}{!}{\includegraphics{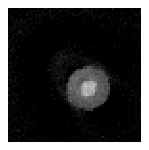}}& \resizebox*{3.5cm}{!}{\includegraphics{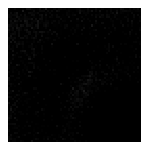}}& \resizebox*{0.8cm}{!}{\includegraphics{barra.eps}}\\

\end{tabular}
\end{center}
\caption{\label{fig:20dB} Reconstructions from noisy data with SNR=20 dB (relative level of noise 10 $\%$). Errors are NMSE=0.40 $\%$, NMAE=4.35 $\%$. Regularization parameter $\lambda=0.05$.} 
\end{figure}

\begin{figure}
\begin{center}
\begin{tabular}{ccccc}

\resizebox*{3.5cm}{!}{\includegraphics{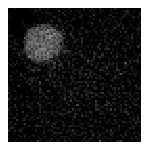}} & \resizebox*{3.5cm}{!}{\includegraphics{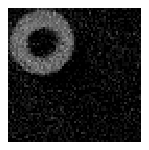}} &
\resizebox*{3.5cm}{!}{\includegraphics{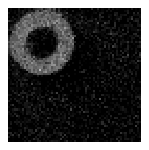}}& \resizebox*{3.5cm}{!}{\includegraphics{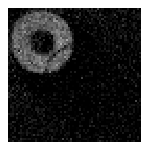}}& \resizebox*{0.8cm}{!}{\includegraphics{barra.eps}}\\

\resizebox*{3.5cm}{!}{\includegraphics{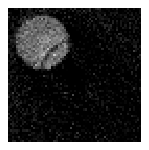}} & \resizebox*{3.5cm}{!}{\includegraphics{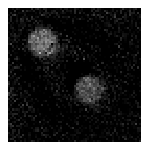}} &
\resizebox*{3.5cm}{!}{\includegraphics{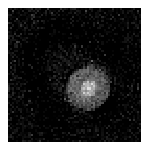}}& \resizebox*{3.5cm}{!}{\includegraphics{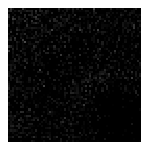}}& \resizebox*{0.8cm}{!}{\includegraphics{barra.eps}}\\

\end{tabular}
\end{center}
\caption{\label{fig:10dB} Reconstructions from noisy data with SNR=10 dB (relative level of noise 29 $\%$). Errors are NMSE=0.92 $\%$, NMAE=7.38 $\%$. Regularization parameter $\lambda=0.05$.} 
\end{figure}

\section{Discussion}
\label{sec:discusion}

What follow are some final comments on our work.
 
\begin{itemize}

\item The toric transform we study exhibits rotational invariance and we employed spherical harmonic expansion to address the problem of invertibility. In order to study solution uniqueness, we used a theory for Abel and first kind Volterra equations with kernels having zeros on the diagonal \cite{SMMH2017}. The proof relied on the analysis of the kernel gradient on its diagonal. In other studies on 3D CST \cite{WL2018}, spherical harmonic expansion was used together with an approach to solve weakly singular Volterra equations of the first kind based on the theory in \cite{Bownds76}. 

%
\item Numerical simulations confirm an additional advantage of this modality: the scanning is feasible when the system is smaller than the object under study. In the proposed simulation we use an object whose side was eight times larger than the radius of the detection sphere.
%

\item Since the model neglects attenuation, extra analysis should be carried out when addressing the realistic scenario. When attenuation is significantly low, few photon counts are expected leading to a sparse signal. In such case, the algorithm may be adapted to deal with sparse data. The influence of attenuation is modeled through weighted Radon transforms on tori leading to weighted kernels. Under some assumptions, this mathematical framework may be applied with weighted kernels both in the theoretical treatment as well as in reconstructions. Other works address the reconstruction problem when attenuation is present \cite{RH2018,RH2020}.


\end{itemize}

\section{Conclusion}

We studied a Compton scattering tomography in three dimensions suitable for scanning large objects. This system has interesting advantages such as fixed and unique source and compactness. Measured data was modeled by a toric Radon transform. We demonstrated the invertibility of the toric transform by proving the uniqueness of the solutions of the one dimensional Abel's type equation resulting from its spherical harmonics expansion. In addition, we developed a reconstruction method based on the discrete spherical harmonics expansion and Tikhonov regularization. Numerical simulations confirm the feasibility of the approach. Future challenges emerge such as incorporating attenuation to the model and dealing with sparse data. Some studies on that direction are on the way.

\section*{Acknowledgements}
 J. Cebeiro is supported a CONICET postdoctoral grant ($\#$ 171800). D. Rubio, J. Cebeiro and M. A. Morvidone are partially supported by SOARD-AFOSR (grant number FA9550-18-1-0523). C. Tarpau research work is supported by grants from R\'egion Ile-de-France (in Mathematics and Innovation) 2018-2021 and LabEx MME-DII (Mod\`eles Math\'ematiques et Economiques de la Dynamique, de l'Incertitude et des Interactions) (No. ANR-11-LBX-0023-01).  

\appendix
\section{Spherical Harmonics Expansion}

\label{sec:SH_reduction}

Here, we present the proof lemma \ref{theo:SHE} explaining how to get form \eref{eq:TRT_omega} to \eref{eq:TRT_SH_RT_lm}. First, we define some useful rotation properties.

\subsection{Rotation properties}

\label{sec:SH_rot}
Given $h^{-1}=u(\alpha)a(\beta)$ as introduced in \eref{eq:u_alpha} we define the rotation operator $\Lambda_h$ by its action on a function $f$ as: 
\begin{equation}
\label{eq:F_rotated}
\left(\Lambda_hf\right) \left( r \Theta(\gamma,\psi)  \right) = f \left(r  h^{-1}  \Theta(\gamma,\psi) \right) .
\end{equation}
Applying this linear operator $\Lambda_h$ to a function expanded as in~(\ref{eq:SH_F}) we have
\begin{equation}\label{eq:F_rotated_2}
(\Lambda_hf)(r\Theta(\gamma,\psi))=\sum_{l\in\NN}\sum_{|m|\leq l}f_{lm}(r)(\Lambda_hY_l^m)(\gamma,\psi).
\end{equation}

Properties $(i)$-$(iii)$ will be useful in what follows:

\begin{enumerate}
\item
Any rotated spherical harmonic of degree $l$ is a linear combination of $Y^n_l$ with $|n|\leq l:$
\begin{equation}\label{eq:Y_rotated}
(\Lambda_hY_l^m)(\gamma,\psi)= \sum_{|n|\leq l}Y_l^n(\gamma,\psi) D_{n,m}^{(l)}(h),
\end{equation}
\item  Matrices $D^{(l)}(h)$ verify: 
\begin{equation} \label{eq:SH_prop1} D_{0,m}^{(l)} (h)=\overline { D_{m,0}^{(l)}}  (h^{-1}). \end{equation}
\noindent
\item From the definition of the matrix entries $ D_{n,m}^{(l)}(h)$ \cite{BL81}, it follows that:
\begin{equation} 
\label{eq:SH_prop2}
Y_l^m(\alpha,\beta)= \sqrt{\frac{2l+1}{4\pi}} \overline {D^{(l)}_{m,0}}(h^{-1}).   
\end{equation}
\noindent
\end{enumerate}
See \cite{BL81} and \cite{DH1994} for details.

\subsection{Proof of lemma \ref{theo:SHE}}

\begin{proof}

From \eref{eq:TRT_omega} and \eref{eq:F_rotated} we have:

\begin{equation}\label{eq:TRT_SH_1}
\mathcal{R}_\mathcal{T}f(\alpha,\beta,\omega)=  \int_0^{2\omega-\pi} d\gamma  \, \int_0^{2\pi} d\psi \, r^\omega(\gamma) \frac{\sin \gamma}{\sin \omega} \,  \left(\Lambda _hf \right)  \left(  r^\omega(\gamma)\Theta(\gamma,\psi \right))  .
\end{equation}

\noindent
Using \eref{eq:F_rotated_2} and \eref{eq:Y_rotated}, we can write:

\begin{eqnarray}\label{eq:TRT_SH_2}
\fl \RR_\TT f(\alpha,\beta,\omega)={  \sum_{l\in\NN}\sum_{|m|\leq l}}\int_0^{2\omega-\pi} d\gamma  \,\int_0^{2\pi} d\psi \,r^\omega(\gamma) \frac{\sin \gamma}{\sin \omega} \, f_{lm}  \left(  r^\omega(\gamma) \right)  (\Lambda_hY_{l}^m)  \left( \gamma,\psi\right) \nonumber\\
\fl =  {  \sum_{l\in\NN}\sum_{|m|\leq l} \sum_{|n|\leq l} }
\int_0^{2\omega-\pi} d\gamma  \,\int_0^{2\pi} d\psi \, r^\omega(\gamma) \frac{\sin \gamma}{\sin \omega} \, f_{lm}  \left(  r^\omega(\gamma) \right)  Y_{l}^n\left( \gamma,\psi\right)  D^{(l)}_{n,m}(h).\nonumber\\
\end{eqnarray}

\noindent
Now, using \eref{eq:SH_def} and performing the integration on variable $\psi$, we have that the only non-vanishing term is the $n=0$ term, so:

\begin{eqnarray}\label{eq:TRT_SH_3}
\fl  \RR_\TT f(\alpha,\beta,\omega)=  {  \sum_{l\in\NN}\sum_{|m|\leq l} \sum_{|n|\leq l} } \nonumber\\
\fl \int_0^{2\omega-\pi} d\gamma  \,\int_0^{2\pi} d\psi \, r^\omega(\gamma) \frac{\sin \gamma}{\sin \omega} \, f_{lm}  \left(  r^\omega(\gamma) \right)  (-1)^n\sqrt{\frac{(2l+1)(l-n)!}{4\pi(l+n)!}}P^n_l(\cos\gamma)e^{i n\psi}  D^{(l)}_{n,m}(h)     \nonumber\\
\fl=  {  \sum_{l\in\NN}\sum_{|m|\leq l}}
2\pi \int_0^{2\omega-\pi} d\gamma r^\omega(\gamma) \frac{\sin \gamma}{\sin \omega} \, f_{lm}  \left( r^\omega(\gamma) \right) P^0_l(\cos\gamma)  \sqrt{\frac{(2l+1) }{4\pi }}   D_{0,m}^{(l)}(h)   \nonumber\\
\fl= {  \sum_{l\in\NN}\sum_{|m|\leq l}} \sqrt{\frac{(2l+1) }{4\pi }}   D_{0,m}^{(l)}(h)  
2\pi \int_0^{2\omega-\pi} d\gamma r^\omega(\gamma) \frac{\sin \gamma}{\sin \omega} \, f_{lm}  \left(  r^\omega(\gamma) \right) P^0_l(\cos\gamma). 
\end{eqnarray}

\noindent
Then, using properties \eref{eq:SH_prop1} and \eref{eq:SH_prop2} we arrive to
\begin{equation}
\fl \RR_\TT f(\alpha,\beta,\omega)= {  \sum_{l\in\NN}\sum_{|m|\leq l}} Y_l^m(\alpha,\beta)\,\,\,\,
2\pi \int_0^{2\omega-\pi} d\gamma r^\omega(\gamma) \frac{\sin \gamma}{\sin \omega} \, f_{lm}  \left(  r^\omega(\gamma) \right) P^0_l(\cos\gamma)  .
\end{equation}

\noindent
Finally, the factor accompanying the $l,m$-spherical harmonic is the $l,m$-coefficient of the Fourier expansion of $(\RR_\TT f)$ in terms of the scattering angle $\omega$ and we get \eref{eq:TRT_SH_RT_lm}. This completes the proof.

%
%
%
\end{proof}

\section{An Algorithm for Discrete Spherical Harmonics Expansion}
\label{sec:DSHE}

Here, we describe succinctly our implementation of the forward-inverse Discrete Spherical Harmonics Transform (DSHT-IDSHT). The technique is based on the algorithms for the discrete Fourier and Legendre transforms as described in \cite{Basko98,Taguchi_2001,DH1994}. Let's be the discrete function $\textbf{F}_{kn}^j=F(r_j,\theta_k,\varphi_n)$ sampled with the following parameters $\Delta r=(r_{\max}-r_{\min})/(N_r-1)$ and $\Delta \varphi = 2\pi/(2N+1)$ with their corresponding indices $j= 0,...,N_r-1$ and $n=-N,...,N$. Discrete values of variables for radial and azimuthal variables are $r_j=j\Delta r + r_{\min}$ and $\varphi_n= 2\pi n/(2N+1)$. Latitude variable $\theta$ can be arbitrarily sampled in the interval $[0,\pi]$ and is labeled $\theta_k$, $k=1,...,N_\theta$, we used uniform sampling.
The relationship between the discrete function $\textbf{F}_{kn}^j$ and its representation in the domain of discrete spherical harmonics $\mathbf{F}_{lm}^j$ is summarized in the diagram:

\[
 \mathbf{F}_{nk}^j \mathrel{\mathop{\rightleftarrows}^{\mathrm{DSHT}}_{\mathrm{IDSHT}}}  \mathbf{F}_{lm}^j \,\,\,\, \,\,  :
\,\,\,\, \,\,  \left\lbrace \textbf{F}_{kn}^j \right\rbrace  \mathrel{\mathop{\rightleftarrows}^{\case{1}{N}\cdot\mathrm{DFT}}_{N\cdot\mathrm{IDFT}}} \left\lbrace  {\textbf{ F}}_{km}^j\right\rbrace   \mathrel{\mathop{\rightleftarrows}^{\mathrm{IDLT}}_{\mathrm{DLT}}} \left\lbrace \textbf{F}_{lm}^j \right\rbrace, \]

\noindent
while the first block is given by the well known discrete Fourier pairs (DFT-IDFT), the second one is given by the associated Discrete Legendre Transform pairs (DLT-IDLT)

\begin{equation}\label{eq:DLT}
\textbf{F}_{km}^j = \mathrm{DLT}^{m} \left\lbrace  \textbf{F}_{lm}^j  \right\rbrace    =     \sum_{l=|m|}^{N}  \textbf{F}_{lm}^j q_l^m P_l^m(t_k),   
\end{equation}

\noindent
and

\begin{equation}\label{eq:IDLT}
 \textbf{F}_{lm}^j = \mathrm{IDLT}^{m} \left\lbrace \textbf{F}_{km}^j \right\rbrace  =    \sum_{k=1}^{N_\theta} \textbf{F}_{km}^j q_l^m P_l^m(t_k)w_k,
\end{equation}

\noindent
where $m=-N,...,N$, $ l=|m|,...,N$, $t_k=\cos \theta_k$, $w_k$ are the Gaussian quadrature coefficients and 
\begin{equation}
\label{eq:qlm}
q_l^m= (-1)^m \sqrt{\case{2l+1}{4\pi}\case{(l-m)!}{(l+m)!}}.
\end{equation}

%
%
%
%
%
%
%
%
%
%
%
%
%
%
%

An alternative implementation can be found in \cite{DH1994}. A preliminary version of the algorithm can be found in the code repository \cite{Cebeiro_Algorithm_for_Discrete_2020}.

\section*{References}
\bibliographystyle{iopart-num-long}
\bibliography{IEEEabrv.bib,bibliography_IP_march_2022.bib}
\end{document}